%% file: main_en_arxive_ver_2.tex
\documentclass[reqno, 12pt,a4paper]{amsart}

\usepackage[T2A]{fontenc}
\usepackage[utf8]{inputenc}
\usepackage[english]{babel}

\usepackage{amsfonts,amssymb,amsthm,mathrsfs,amscd}
\usepackage{graphicx}

\theoremstyle{plain}
\newtheorem{theorem}{Theorem}
\newtheorem{lemma}{Lemma}

\newtheorem{proposition}{Proposition}

\theoremstyle{definition}

\newtheorem{remark}{Remark}

\input{renames.tex}
\begin{document}

\title{Monotonicity of Discrete spectra of Dirichlet Laplacian in 3-dimensional layers}
% Если колонтитул совпадает с полным названием статьи, то можно писать короче
% \title{Полное название статьи}

\author[Bakharev F. L.]{Bakharev F. L.${}^\dagger$}

\thanks{\textsc{${}^\dagger$}{Chebyshev Laboratory, St. Petersburg State University, 14th Line V.O., 29, Saint Petersburg 199178 Russia}}

\email{fbakharev@yandex.ru}

\author[Matveenko S. G.]{Matveenko S. G.${}^\dagger$}

\email{matveis239@gmail.com}

\thanks{{\it E-mail addresses:} \texttt{fbakharev@yandex.ru, matveis239@gmail.com}}

\keywords{Laplace operator, Dirichlet layers, discrete spectrum, continuous spectrum}

\begin{abstract}
We investigate monotonicity properties of eigenvalues of the Dirichlet Laplacian in polyhedral layers of fixed width. We establish that eigenvalues below the essential spectrum threshold monotonically depend on geometric parameters defining the polyhedral layer, generalizing previous results known for planar V-shaped waveguides and conical layers. Moreover, we demonstrate non-monotone spectral behavior arising from asymmetric geometric perturbations, providing an explicit example where unfolding the polyhedral layer unexpectedly leads to the emergence of discrete eigenvalues.
The limiting behavior of eigenvalues as the geometric parameters approach critical configurations is also rigorously analyzed.
\end{abstract}

\thanks{The work is supported by the Russian Science Foundation grant 19-71-30002}

\maketitle

\section{Introduction}

In the 1990s, the experimental realization of low-dimensional semiconductor structures (e.g., quantum wires and thin films) revealed 
electron transport phenomena in nanoscale devices, where quantum mechanical effects dominate classical considerations.
It was established that the appearance of trapped modes depends on the geometric structure of waveguides.
The influential early result about the existence and uniqueness of the eigenvalue below the essential spectrum in right-angled $V$-shaped planar waveguides was obtained in \cite{ExSeSt89}. 
In~\cite{AvBeGiMa91}, coordinate transformations to analyze multiplicity of discrete spectrum were introduced (see also \cite{CaLoMuMu93} for the experimental confirmation of the theoretical results). Later in \cite{DaLaRa12} and  \cite{DaRa12} eigenvalue monotonicity  was rigorously demonstrated for geometrically symmetric domains, such as planar V-shaped waveguides.
Further significant contributions were made in~\cite{Pa17}, where  estimates for the angular openings that ensure eigenvalue uniqueness were provided, refining earlier results from~\cite{NaSh14}. The latter work also includes multiplicity estimates for eigenvalues associated with small angles.

Parallel developments occurred in the study of Dirichlet layers --  unbounded regions with constant thickness constructed around geometric configurations in three-dimensional space (see \cite{DuExKr01, CaExKr04, LiLu07}). Despite their richer geometric complexity, Dirichlet layers exhibit spectral properties closely resembling those of quantum waveguides, particularly regarding the nature of their essential spectra and the appearance of discrete spectra associated with trapped modes (see more details about polyhedral layers in~\cite{DaLaOu18, BaAI20, BaMa24}, about conical layers in  \cite{CaExKr04, ExTa10, DaOuRa15}).

In this paper, we extend these studies to the setting of three-dimensional polyhedral layers. Specifically, we investigate constant-thickness Dirichlet layers constructed from polyhedral angles, emphasizing the monotonic dependence of discrete eigenvalues on geometric parameters. We employ coordinate transformations similar to those proposed in~\cite{AvBeGiMa91}; however, rather than solely establishing the existence or multiplicity of eigenvalues, our focus is on rigorously proving their monotonicity. Our results thus generalize previous monotonicity proofs (such as those in \cite{DaLaRa12, DaRa12}) to cases lacking geometric symmetry.

We begin by establishing proofs for known results concerning fundamental configurations, including planar V-shaped waveguides and conical layers and derive monotonic relationships between geometric parameters and eigenvalues beneath the essential spectrum threshold. Subsequently, we explore general polyhedral layers, providing rigorous results on eigenvalue monotonicity and their limiting behavior as geometric parameters approach critical values. Additionally, we illustrate that asymmetric unfolding of polyhedral angles can lead to non-monotonic behavior of the number of eigenvalues, exemplified by situations where eigenvalues appear or vanish unexpectedly as geometric parameters vary.

The paper is structured as follows.  
Section~\ref{sec-wg} reviews known results about spectral properties of planar V-shaped waveguides and contains key ideas for the later sections. 
In Section~\ref{sec-conical} the same spectral questions arise for conical layers. In Section~\ref{sec-2} we rigorously prove spectral monotonicity results for general polyhedral layers. In Section~\ref{sec-6}, we demonstrate a noteworthy example of non-monotonic spectral behavior arising from geometric asymmetry. 

\section{V-shaped waveguide}
\label{sec-wg}

Consider a planar angle of magnitude $2\theta$, $\theta\in(0, \pi/2)$.  Define a broken waveguide~$\Omega_\theta$ of unit width as a set of points inside this angle whose distance to the boundary is less than one. 
We work in Cartesian coordinates $\pmb{\bfx}=(x, z)$; the vertex is at the origin and the symmetry axis is the $z$-axis (see Fig.~\ref{fig-broken-wg}).

Define the Dirichlet Laplacian $\mathcal{A}^{\Omega}$ generated by a positive closed sesquilinear form 
\[
\mfa_{\Omega}[u, v] = (\nabla u, \nabla v)_{\Omega},
\]
defined on the Sobolev space $\SobNil(\Omega)$ of functions with square-integrable derivatives and zero trace on the boundary $\partial\Omega$.

Let us recall some properties of the Dirichlet spectral problem for the Laplace operator in the domains $\Omega_\theta$. For all $\theta\in (0,\pi/2)$ the essential spectrum of $\mathcal{A}^{\Omega_\theta}$ coincides with the ray $[\pi^2,+\infty)$. Moreover, its discrete spectrum is nonempty, finite, and for all $\theta\in (\arctan(\sqrt 3 / 4), \pi/2)\supset (\pi/4, \pi/2)$ there is a unique eigenvalue $\lambda_1(\mathcal{A}^{\Omega_\theta})$ below the threshold of the continuous spectrum (see, e.g., \cite{ExKo15} and \cite{Pa17}).
In \cite{AvBeGiMa91} the total multiplicity of the discrete spectrum of $\mathcal{A}^{\Omega_\theta}$ is proved to exceed any preassigned number for sufficiently small $\theta$.
Additionally, the first eigenvalue $\lambda_1(\mathcal{A}^{\Omega_\theta})$ is a monotonically increasing function on $(0,\pi/2)$ and $\lambda_1(\mathcal{A}^{\Omega_\theta})\to \pi^2$ as $\theta\to \pi/2$. In~\cite{DaRa12} the asymptotics for sharp angles is obtained, namely, $\lambda_1(\mathcal{A}^{\Omega_\theta}) \to \pi^2/4$ for $\theta\to 0$.

\begin{figure}[ht!]
\begin{center}
    \includegraphics[width=0.55\textwidth]{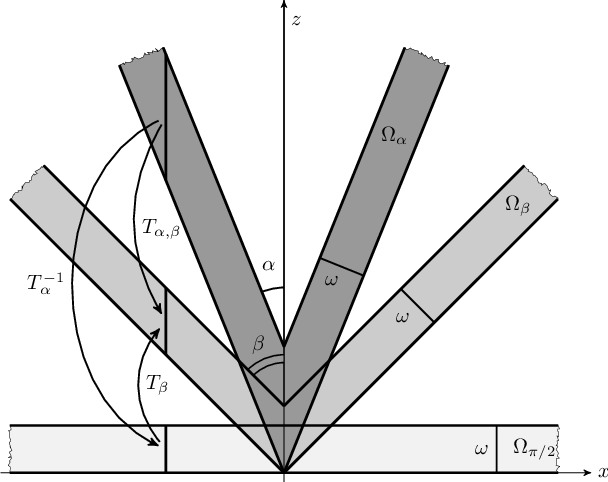}
    \caption{The {\sf V}-shaped waveguide $\Omega_\alpha$ (dark gray) overlaps the {\sf V}-shaped waveguide  $\Omega_\beta$ (light gray) and the straight strip $\Omega_{\pi/2}.$}
    \label{fig-broken-wg}
\end{center}
\end{figure}

The following lemma results from \cite[Theorem 10.2.3]{BiSo80}. 
\begin{lemma}
\label{lem-mutual-ev}
Let $\mfa$ be a closed lower semi-bounded  sesquilinear form  in a Hilbert space~$H$ with domain $\cD(\mfa)$, and let this form generate a self-adjoint operator~$\mathcal{A}$. Suppose that for some $M\leq \infty$ the spectrum of $\mathcal{A}$ below $M$ is discrete. If there exist linearly independent functions $v_1, v_2,\ldots, v_k\in\cD(\mfa)$ such that for any function $v$ from their linear span the following
inequality holds true 
\begin{equation*}
%\label{lem1-form-bound}
\mfa[v, v] < M\|v;H\|^2,
\end{equation*}
then the total multiplicity of the spectrum below $M$ is at least $k$.
\end{lemma}

Define an isomorphism between the straight strip $\Omega_{\pi/2}$ and the broken one~$\Omega_\beta$
(see Fig.~\ref{fig-broken-wg})
\begin{equation*}
T_\beta\colon \Omega_{\pi/2}\to\Omega_\beta,\qquad
T_\beta(\pmb{\bf x})=(x, z\csc\beta+|x|\cot\beta),    
\end{equation*}
and its inverse
\begin{equation*}
 T_\beta^{-1}\colon \Omega_\beta\to\Omega_{\pi/2},\qquad
T_\beta^{-1}(\pmb{\bf x})=(x, z\sin\beta-|x|\cos\beta).
\end{equation*}

Let $0<\alpha<\beta<\pi/2$. We define a continuous piecewise-linear bijective map $T_{\alpha,\beta} = T_\beta\circ T_\alpha^{-1}:\Omega_\alpha\to \Omega_\beta$ (see Fig.~\ref{fig-broken-wg}).  

\begin{lemma} 
\label{lem-q-form-monotony}
    Suppose there exist $\beta\in(0,\pi/2)$ and a function $u\in\SobNil(\Omega_\beta)$, normalized in $L_2(\Omega_\beta)$, such that $\mfa_{\Omega_\beta}[u,u]<\pi^2$. Then for all $\alpha\in(0,\beta)$ the functions $v_\alpha \in\SobNil(\Omega_\alpha)$ defined by
% \begin{equation}
%  \label{lem-q-form-monotony-func-trans}
% v_\alpha = \left(\frac{\sin\alpha}{\sin\beta}\right)^{1/2}u \circ T_{\alpha,\beta}
% \end{equation}
\begin{equation}
\label{lem-q-form-monotony-func-trans}
v_\alpha = \left(\frac{\sin\alpha}{\sin\beta}\right)^{1/2}u \circ T_{\alpha,\beta}=:\mathcal{T}_{\alpha, \beta}u,
\end{equation}
are normalized in $L_2(\Omega_\alpha)$. Furthermore, the function $\alpha \mapsto \mfa_{\Omega_\alpha}[v_\alpha, v_\alpha]$ increases monotonically. In particular, $\mfa_{\Omega_\alpha}[v_\alpha, v_\alpha] < \pi^2$ for all $\alpha \in (0, \beta)$.
\end{lemma}
\begin{proof} 
The equality $\|v_\alpha; L_2(\Omega_\alpha)\| = 1$ follows immediately from the fact that the Jacobian of $T_{\alpha,\beta}$ is constant and equals $\frac{\sin\alpha}{\sin\beta}$.
Note that, being Lipschitz, isomorphisms~$T_\alpha$ and their inverses $T_\alpha^{-1}$ preserve the Sobolev spaces. Hence, the functions $v_{\alpha}$ belong to~$\SobNil(\Omega_\alpha)$. 

Applying the change of variables $\pmb{\bf x}\mapsto T_\beta(\pmb{\bf x})$, we transfer $u$ into the straight strip, obtaining 
 \begin{equation*}
    \mfa_{\Omega_\beta}[u, u] = \mfa_{\Omega_{\pi/2}}[v_{\pi/2}, v_{\pi/2}] + \cos\beta\cdot {\mathfrak r}(v_{\pi/2})
\quad \mbox{where} \quad
 {\mathfrak r} (v) = 
- 2  (\mathop{\rm sgn}\nolimits(x) \partial_z v, \partial_x v)_{\Omega_{\pi/2}}.
\end{equation*}
It should be noted that by definition~\eqref{lem-q-form-monotony-func-trans} the function $v_{\pi/2}$ depends on $\beta$.
Since the discrete spectrum of the Dirichlet Laplacian in the straight strip $\Omega_{\pi/2}$ is empty, according to the max-min principle we have  
$\mfa_{\Omega_{\pi/2}}[v_{\pi/2}, v_{\pi/2}] \geqslant \pi^2$.
Thus, if $\mfa_{\Omega_\beta}[u,u]<\pi^2$ it follows necessarily that ${\mathfrak r}(v_{\pi/2})<0$.
Therefore, the inequality 
 \begin{equation*}
    \mfa_{\Omega_\beta}[u, u] > \mfa_{\Omega_{\pi/2}}[v_{\pi/2}, v_{\pi/2}] + \cos\alpha\cdot {\mathfrak r}(v_{\pi/2})
\end{equation*}
holds for all $\alpha\in(0, \beta)$. Applying the inverse change of variables $\pmb{\bf x}\mapsto T_\alpha^{-1}(\pmb{\bf x})$ on the right-hand side transfers $v_{\pi/2}$ back into the broken waveguide $\Omega_\alpha$. This yields the inequality
\begin{equation*}
    \mfa_{\Omega_\beta}[u, u] > \mfa_{\Omega_\alpha}[v_{\alpha}, v_\alpha]
\end{equation*}
 which completes the proof.
\end{proof}

\begin{theorem} \label{thm-monotonicity}
    If ${\mathcal{A}}^{\Omega_\beta}$ has $k$ eigenvalues below the threshold for some $\beta\in(0,\pi/2)$, then the operator ${\mathcal{A}}^{\Omega_\alpha}$ has at least $k$ eigenvalues (counted with multiplicities) below the threshold for all $\alpha\in (0, \beta)$. Moreover, the $j$-th eigenvalue $\lambda_j({\mathcal{A}}^{\Omega_\alpha})$ is a monotonically increasing function with respect to $\alpha$ for all $1\leqslant j\leqslant k$. 
\end{theorem}
\begin{proof} 
Suppose the functions $u_{1}, u_{2}, \ldots, u_{k}$ form an orthonormal set in  $L_2(\Omega_\beta)$, consisting of eigenfunctions of ${\mathcal{A}}^{\Omega_\beta}$ corresponding to the eigenvalues $\lambda_{1}({\mathcal{A}}^{\Omega_\beta})$, $\lambda_{2}({\mathcal{A}}^{\Omega_\beta}), \ldots,$ $\lambda_{k}({\mathcal{A}}^{\Omega_\beta})$ respectively. 
Consider a linear operator 
\begin{equation*}
\cT_{\alpha, \beta}\colon \SobNil(\Omega_\beta)\mapsto\SobNil(\Omega_\alpha),
\end{equation*} given by formula~\eqref{lem-q-form-monotony-func-trans}.
%$\cT_{\alpha, \beta} u = \left(\frac{\sin\alpha}{\sin\beta}\right)^{1/2}u\circ T_{\alpha,\beta}$. 
This operator is one-to-one and onto, since it has an inverse~$\cT_{\beta, \alpha}$. Moreover, due to Lemma~\ref{lem-q-form-monotony} it preserves $L_2$-norm.
Thus, the functions 
%$v_{j,\alpha} = \left(\frac{\sin\alpha}{\sin\beta}\right)^{1/2}u_{j}\circ T_{\alpha,\beta}$,
$v_{j,\alpha} = \cT_{\alpha,\beta} u_j$, 
$1\leqslant j\leqslant k$, belong to $\SobNil(\Omega_\alpha)$ and remain orthonormal in $L_2(\Omega_\alpha)$ and hence linearly independent. 
Furthermore, by Lemma~\ref{lem-q-form-monotony} the chain of inequalities
\begin{equation*}
%\label{thm2-form-ineq}
\mfa_{\Omega_\alpha}[\cT_{\alpha,\beta} u, \cT_{\alpha,\beta} u] < \mfa_{\Omega_\beta}[u, u] \leq \pi^2\|u; L_2(\Omega_\beta)\|^2 = \pi^2\|\cT_{\alpha,\beta} u; L_2(\Omega_\alpha)\|^2
\end{equation*}
holds true for every $u$ from linear span of the eigenfunctions $u_{1}, u_{2}, \ldots, u_{k}$. Therefore, the existence of $k$ eigenvalues below the threshold for $\alpha\in (0,\beta)$ and their monotonicity in $\alpha$ immediately follow from Lemma \ref{lem-mutual-ev}.
\end{proof}

\begin{remark}
\label{rem-first-ev-as}
    The first eigenvalue $\lambda_1(\mathcal{A}^{\Omega_\alpha})$ converges to the threshold value $\pi^2$ as~\mbox{$\alpha\to\pi/2$}. 
\end{remark}
\begin{proof}
Let $u$ be an eigenfunction normalized in $L_2(\Omega_\alpha)$ corresponding to $\lambda_1(\mathcal{A}^{\Omega_\alpha})$. Using the notation from the proof of Lemma \ref{lem-q-form-monotony} we have
 \begin{equation*}
    \mfa_{\Omega_\alpha}[u,u] = \mfa_{\Omega_{\pi/2}}[v_{\pi/2},v_{\pi/2}] + \cos\alpha\cdot {\mathfrak r}(v_{\pi/2}).
\end{equation*}
Applying the arithmetic–geometric mean inequality
\begin{equation*}
    2|\partial_x v_{\pi/2}(\pmb{\bf x})\partial_z v_{\pi/2}(\pmb{\bf x})| \leqslant
    |\nabla v_{\pi/2}(\pmb{\bf x})|^2,
\end{equation*}
we obtain the estimate
\begin{equation*}
     \mfa_{\Omega_\alpha}[u,u] \geqslant (1 - \cos\alpha)\mfa_{\Omega_{\pi/2}}[v_{\pi/2},v_{\pi/2}].
\end{equation*}
Keeping in mind the inequality $\lambda_1(\mathcal{A}^{\Omega_\alpha}) = \mfa_{\Omega_\alpha}[u, u] <\pi^2$ and using the max-min principle we obtain the following chain of estimates
\begin{equation*}
    \pi^2>\lambda_1(\mathcal{A}^{\Omega_\alpha})\geqslant (1 - \cos\alpha)\mfa_{\Omega_{\pi/2}}[v_{\pi/2},v_{\pi/2}]\geqslant(1-\cos\alpha)\pi^2.
\end{equation*}
Since $\cos\alpha\to0$ as $\alpha\to\pi/2$, the squeeze theorem yields the convergence of the first eigenvalue $\lambda_1(\mathcal{A}^{\Omega_\alpha})$ to $\pi^2$.
\end{proof}

\begin{remark}
The results presented in this paper remain valid not only for plane V-shaped waveguides with unit segment cross-sections at their outlets but also for multidimensional waveguides whose cross-sections are connected domains with a Lipschitz boundary (see, e.g., \cite{BaMa25} for details). 
\end{remark}

\section{Conical layer}
\label{sec-conical}

\def\Cone{\Pi}
A natural generalization of a broken waveguide is a circular conical layer. Consider an infinite right circular cone in three-dimensional space with vertex angle $2\theta$, defined such that the angle between the cone's generatrix (lateral boundary) and its axis of symmetry equals $\theta$.  
A conical layer $\Cone^\theta$ of unit thickness is a region inside this cone consisting of all points whose distance from the cone’s lateral surface is less than one (see Fig.~\ref{fig-con-lay}).
With this notation, the straight layer (the region of space bounded by two parallel planes at a distance of one) corresponds to $\Cone^{\pi/2}$. 

We use a cylindrical coordinate system $\pmb{\bfx} = (r, \varphi, z)$ with $r\geq0$ and $\varphi\in[0, 2\pi)$ placing the apex of the cone at the origin and aligning its symmetry axis with the longitudinal $z$-axis.  

\begin{figure}[ht!]
\begin{center}
    \includegraphics[width=0.3\linewidth]{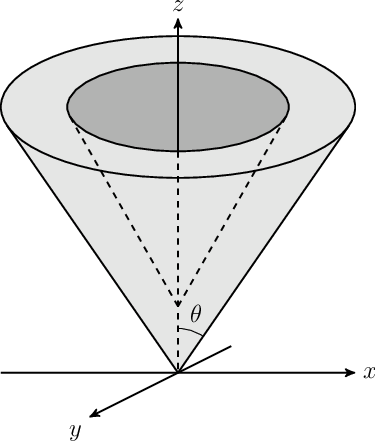}
    \caption{Conical layer with vertex angle $2\theta$}
\label{fig-con-lay}
\end{center}
\end{figure}

The following claim is standard and follows from separation of variables (Fourier transform in longitudinal variable).

\begin{proposition}
	\label{thm-sp-in-straight-layer}
	The spectrum of the operator ${\mathcal{A}^{\Cone^{\pi/2}}}$ coincides with the ray
	\[
	\sigma(\mathcal{A}^{\Cone^{\pi/2}}) = \sigma_{ess}(\mathcal{A}^{\Cone^{\pi/2}}) = [\pi^2,+\infty).
	\]
\end{proposition}

As in Section \ref{sec-wg}, we define an isomorphism between the straight layer $\Cone^{\pi/2}$ and the conical one~$\Cone^\theta$:
\begin{equation*}
T_\theta\colon \Cone^{\pi/2}\to\Cone^\theta,\qquad
T_\theta(\pmb{\bf x})=( r, \varphi, z\csc\theta+r\cot\theta)
\end{equation*}
and its inverse:
\begin{equation*}
 T_\theta^{-1}\colon \Cone^\theta\to\Cone^{\pi/2},\qquad
T_\theta^{-1}(\pmb{\bf x})=(r, \varphi, z\sin\theta-r\cos\theta).
\end{equation*}
For $0<\theta<\varsigma<\pi/2$, we define a continuous piecewise-linear bijection $T_{\theta,\varsigma} = T_\varsigma\circ T_\theta^{-1}$.

The same reasoning as in Lemma \ref{lem-q-form-monotony} demonstrates that the function 
$$
\theta \mapsto \mfa_{\Cone^\theta}[v_\theta,v_\theta] \quad \mbox{where} \quad
v_\theta = \left(\frac{\sin\theta}{\sin\varsigma}\right)^{1/2}u \circ T_{\theta,\varsigma},
$$ is monotonically increasing (here  $u\in\SobNil(\Pi^\varsigma)$, normalized in $L_2(\Pi^\varsigma)$, and satisfy the inequality $\mfa_{\Pi^\varsigma}[u,u]<\pi^2$). Since, as in case of the broken waveguide, the threshold value equals $\pi^2$, the only change required in the proof is the explicit form of the remainder term:
\begin{equation*}
 {\mathfrak r} (v) = 
- 2  (\partial_z v, r\partial_r v)_{\Cone^{\pi/2}}.
\end{equation*}

The main difference between the spectral problem for a conical layer and the analogous problem for a broken waveguide is that the discrete spectrum of $\mathcal{A}^{\Pi^{\theta}}$ contains infinitely many eigenvalues below the threshold (see~\cite{DaOuRa15}). Nevertheless, applying Lemma~\ref{lem-mutual-ev}, we repeat the proof of Theorem \ref{thm-monotonicity} step-by-step and arrive at the following result:
\begin{theorem} %\label{thm-cone-monotonicity}
    Let $\{\lambda_j(\mathcal{A}^{\Pi^{\varsigma}})\}_{j=1}^\infty$ be the eigenvalues of $\mathcal{A}^{\Pi^{\varsigma}}$ below the threshold $\Lambda_\dagger=\pi^2$. Then the eigenvalues  $\{\lambda_j(\mathcal{A}^{\Pi^{\theta}})\}_{j=1}^\infty$ of $\mathcal{A}^{\Pi^{\theta}}$ satisfy the inequalities
    \[
    \lambda_j(\mathcal{A}^{\Pi^{\theta}}) < \lambda_j(\mathcal{A}^{\Pi^{\varsigma}})
    \]
    for all $\theta\in (0, \varsigma)$. Moreover, each eigenvalue  $\lambda_j(\mathcal{A}^{\Pi^{\theta}})$ is monotonically increasing as a function of $\theta$. 
\end{theorem}

Repeating the arguments from the proof of Remark \ref{rem-first-ev-as}, we obtain 
\begin{remark}
%\label{rem-ev-as-cone}
    Each eigenvalue $\lambda_j(\mathcal{A}^{\Cone^\theta})$, $j\geq 1$, converges to the threshold value $\pi^2$ as $\theta\to\pi/2$. 
\end{remark}

\section{Polyhedral layer}
\label{sec-2}
\subsection{Domain geometry}
%\label{sub-sec-2.1}

We consider a convex $n$-gon $P_1P_2\ldots P_n$ ($n\geq 3$), lying in the plane $z=1$, in the space $\mathbb{R}^3$, equipped with the Cartesian coordinate system $(x, y, z)$ with its origin $O$. The convex hull of the rays $\ell_j=OP_j$ forms a polyhedral angle with $n$ faces
$$\Upsilon = \mathop{\rm conv}\nolimits\{ \ell_j \colon 1\leq j\leq n\}.$$ 
Its boundary $\Gamma=\partial \Upsilon$ consists of vertex angles $\Gamma_j=\mathop{\rm conv}\nolimits\{\ell_j, \ell_{j+1}\}$, which are equal to $\alpha_j=\angle P_{j}OP_{j+1}$. The inner dihedral angle along the edge $\ell_j$ is denoted by~$\beta_j$.
Here and below we assume that indices are taken modulo $n$, i.e. $n+1=1$, $n+2=2$, and so on.  

\begin{figure}[ht!]
    \centering
    \includegraphics[width=0.35\textwidth]{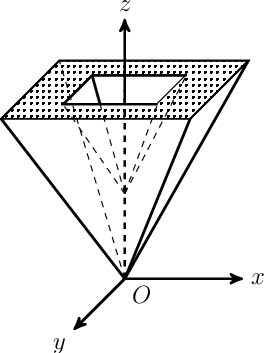}  
    \qquad\qquad\qquad
    \includegraphics[width=0.29\textwidth]{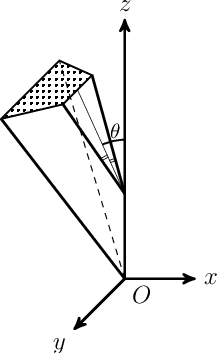} 
    \caption{Layer built on regular tetrahedral angle on the left and an element $\varpi^\theta_{4,2}$ of the partition on the right (angles marked by double arcs are equal to $\alpha/2$) }
    \label{fig-quad-layer}
\end{figure}

The polyhedral layer of unit width is defined by the formula 
$$
\Pi=\{\bfx \in \Upsilon \colon \mathop{\rm dist}\nolimits(\bfx, \Gamma)\in (0,1)\}.
$$
An example of a layer constructed from the quadrihedral angle is represented in Fig.~\ref{fig-quad-layer}, left.
Its boundary $\partial \Pi$ consists of two polyhedral surfaces: the outer one coincides with~$\Gamma$ and the inner one is denoted by $\Gamma'$. If a ball can be inscribed into the angle~$\Upsilon$, then $\Gamma$ and $\Gamma'$ can be matched by a parallel shift. Otherwise, the inner polyhedral surface $\Gamma'$ may have more complicated structure. In this work, however, we consider only the first situation, which occurs, for instance, in the case of any trihedral angle, as well as in the case of a regular polyhedral angle.
We call a polyhedral layer \textit{regular} if its outer boundary $\Gamma$ forms a regular polyhedral cone, i.e. all dihedral angles $\beta_j$ are equal and all vertex angles $\alpha_j$ at the vertex $O$ are equal. In what follows we denote the corresponding quantities by $\beta$ and~$\alpha$ respectively.

\subsection{Spectral problem and preliminary information about the spectrum}

In recent papers \cite{DaLaOu18}, \cite{BaAI20}, and \cite{BaMa24}
it was established that the essential spectrum of the operator $\mathcal{A}^\Pi$ occupies the ray $[\Lambda_\dagger,+\infty)$. The cut-off value $\Lambda_\dagger=\Lambda_\dagger(\Pi)$ equals the first eigenvalue of the ${\sf V}$-shaped waveguide whose inner vertex angle matches the smallest of the dihedral angles $\beta_j$.
In \cite{BaMa24} the authors proved that the discrete spectrum of the operator $\mathcal{A}^\Pi$ in a polyhedral layer $\Pi$ is finite and lies below the threshold $\Lambda_\dagger$. Moreover, for regular polyhedral layers the discrete spectrum is not empty.

The eigenvalues of the operator $\mathcal{A}^\Pi$ below the threshold $\Lambda_\dagger$ can be evaluated using the max-min principle (see \cite[Theorem 10.2.2]{BiSo80}):
\[
\lambda_j(\Pi)\ \ =\sup\limits_{\mathop{\rm codim}\nolimits E=j-1}\ \inf\limits_{u\in E\setminus{\{\vec{0}\}}}\ 
\frac{\|\nabla u;L_2(\Pi)\|^2}{\|u;L_2(\Pi)\|^2},
\] 
where the supremum is taken over all subspaces~$E$ of codimension~ $j-1$ in the space~$\lefteqn{\overset\circ{\hphantom{H'}\vphantom{\prime}}}H\vphantom{H}^1(\Pi)$. 
If the first eigenvalue $\lambda_1(\Pi)$ of the Dirichlet problem exists, it is known to be simple, and the corresponding eigenfunction can be chosen to be positive.

We use a cylindrical coordinate system $\pmb\bfx = (r, \varphi, z)$ with the origin at $O$ and the axis of rotational symmetry coinciding with the $z$-axis. We denote the angle of rotational symmetry by 
\[
\psi_N = \frac{2\pi}N.
\] 
Denote by $\Pi_N^\theta$ a regular polyhedral layer whose outer boundary has  $N$ dihedral angles, with the angle between each face and the axis of rotational symmetry equal to $\theta$. This angle $\theta$ can be expressed in terms of the vertex angle $\alpha$ and the angle of rotational symmetry:
\[
\sin\theta = \tan\left(\frac \alpha 2\right)\cot\left(\frac{\psi_N}{2}\right).
\]

Consider the following partition of $\Pi^\theta_N$ (see Fig.~\ref{fig-quad-layer}, right):
\[
\varpi^\theta_{N,j} = \left\{\pmb\bfx = (r, \varphi, z)\in\Pi^\theta_{N}\colon \varphi\in\left(j\psi_N-\frac{\pi}{N}, j\psi_N+\frac\pi N\right)\right\},\qquad 0\leq j \leq N-1.
\]
The quadratic form for the operator $\mathcal{A}^{\Pi_N^\theta}$ can be written as
\[
\mfa_{\Pi^\theta_N}[u, u] = \sum_{j=0}^{N-1}\int_{\varpi_{N,j}^\theta}\left(|\partial_ru|^2+r^{-2}|\partial_\varphi u|^2+|\partial_zu|^2\right)r\,dr\,d\varphi\,dz.
\]

Again, we define an isomorphism between the straight layer $\Pi^{\pi/2}_N$ and the polyhedral one $\Pi^\theta_N$ and its inverse:
\begin{eqnarray*}
T_\theta\colon \Pi^{\pi/2}_N\to\Pi^\theta_N,&
T_\theta(\pmb{\bf x})=(r, \varphi, z\csc\theta+r\cot\theta\cos(\varphi - j\psi_N)),& \pmb{\bf x}\in\varpi_{N,j}^{\pi/2}, \\
 T_\theta^{-1}\colon \Pi^\theta_N\to\Pi^{\pi/2}_N,&
T_\theta^{-1}(\pmb{\bf x})=(r, \varphi, z\sin\theta-r\cos\theta\cos(\varphi - j\psi_N)),& \pmb{\bf x}\in\varpi_{N,j}^{\theta}.
\end{eqnarray*}
For $0<\theta<\varsigma<\pi/2$ we define a continuous, piecewise-linear bijection $T_{\theta,\varsigma} = T_\varsigma\circ T_\theta^{-1}$.

Below we prove an analog of Lemma \ref{lem-q-form-monotony} establishing the monotonicity of the quadratic form $\mfa_{\Pi^\theta_N}$.

\begin{lemma} \label{lem-layer-q-form-monotony}
    Suppose there exist $\varsigma\in(0,\pi/2)$ and $u\in\SobNil(\Pi^{\varsigma}_{N})$, normalized in $L_2(\Pi^{\varsigma}_N)$,  such that 
    %$\mfa_{\Pi^{\varsigma}_N}[u, u]<\Lambda_\dagger$
    $\mfa_{\Pi^{\varsigma}_N}[u, u]<\pi^2$; then for all $\theta\in(0,\varsigma)$
    the functions $v_\theta \in\SobNil(\Pi^\theta_N)$ defined by
    \begin{equation}
    \label{v_theta}
v_\theta = \left(\frac{\sin\theta}{\sin\varsigma}\right)^{1/2}u \circ T_{\theta,\varsigma}=:\mathcal{T}_{\theta, \varsigma} u,
\end{equation}
are also normalized in $L_2(\Pi^\theta_N)$. Moreover, the map $\theta \mapsto \mfa_{\Pi^\theta_N}[v_\theta,v_\theta]$ increases monotonically for all $\theta \in (0, \varsigma)$.
\end{lemma}
\begin{proof} 
The equality $\|v_\theta, L_2(\Pi^\theta_N)\| = 1$ follows immediately from the fact that the Jacobian of $T_{\theta,\varsigma}$ is constant and equals $\frac{\sin\theta}{\sin\varsigma}$.
Since isomorphisms~~$T_{\theta,\varsigma}$  and their inverse~$T_{\theta,\varsigma}^{-1}$ are Lipschitz ones, they preserve the Sobolev spaces. Hence, the functions $v_{\theta}$ belong to the space $H^1(\Pi^\theta_N)$.

Applying the change of variables $\pmb{\bf x}\to T_\varsigma(\pmb{\bf x})$, we obtain
\begin{equation*}
\mfa_{\Pi_N^\varsigma}[u, u] = \mfa_{\Pi^{\pi/2}_N}[v_{\pi/2}, v_{\pi/2}] + \cos\varsigma\cdot {\mathfrak r}(v_{\pi/2}),
\end{equation*}
where
\begin{equation*}
{\mathfrak r} (v) = 
- 2 \sum_{j=0}^{N-1}\int_{\varpi_{N,j}^{\pi/2}} \left(\cos(\varphi-j\psi_N)r\partial_r v \partial_z v + \sin(\varphi-j\psi_N)\partial_\varphi v \partial_z v\right) \,dr\,d\varphi\,dz.
\end{equation*}
According to Proposition \ref{thm-sp-in-straight-layer}, the max-min principle implies that 
$\mfa_{\Pi^{\pi/2}_N}[v_{\pi/2}, v_{\pi/2}] \geqslant \pi^2.$
Since by assumption $\mfa_{\Pi_N^\varsigma}[u,u]<\pi^2$, it follows necessarily that the term ${\mathfrak r}(v_{\pi/2})$ must be negative.
Therefore, the inequality
 \begin{equation*}
    \mfa_{\Pi_N^\varsigma}[u, u] >\mfa_{\Pi_N^{\pi/2}}[v_{\pi/2}, v_{\pi/2}] + \cos\theta\cdot {\mathfrak r}(v_{\pi/2})
\end{equation*}
holds for all $\theta\in(0, \varsigma)$. Applying the inverse change of variables $\pmb{\bf x}\to T_\theta^{-1}(\pmb{\bf x})$ on the right-hand side, we transfer the inequality back to the layer $\Pi_N^\theta$ and arrive at
\begin{equation*}
    \mfa_{\Pi_N^\varsigma}[u, u] > \mfa_{\Pi_N^\theta}[v_{\theta}, v_\theta],
\end{equation*}
which completes the proof.
\end{proof}

Unlike the problems considered previously for broken waveguides and conical layers, the cut-off value of the essential spectrum in the spectral problem for a polyhedral layer depends both on $\theta$ and $N$ (see \cite{BaMa24}).
%the opening angle and on the dihedral angles of layer's boundary surface. 
We denote the threshold value corresponding to $\Pi^\theta_N$ by $\Lambda_\dagger^\theta$. This threshold equals the first eigenvalue of the Dirichlet spectral problem for the Laplacian in a V-shaped waveguide of unit width with its opening angle $\beta$.

\begin{theorem}
\label{thm-layer-regular-mon}
        If each of the operators $\mathcal{A}^{\Pi^{\varsigma}_N}$ and $\mathcal{A}^{\Pi^{\theta}_N}$ has at least $k$ eigenvalues below their essential spectrum thresholds and $0< \theta < \varsigma < \pi/2$, then 
        \[\lambda_j(\mathcal{A}^{\Pi^\theta_N}) \leq \lambda_j(\mathcal{A}^{\Pi^\varsigma_N})\]
    for all $1\leq j\leq k$.
\end{theorem}
\begin{proof}
Assume that the functions $u_{1}, u_{2}, \ldots, u_{k}$ form an orthonormal set in $L_2(\Pi_N^\varsigma)$, consisting of eigenfunctions of $\mathcal{A}^{\Pi_N^\varsigma}$ corresponding to the eigenvalues $\lambda_{1}(\mathcal{A}^{\Pi^\varsigma_N})$, $\lambda_{2}(\mathcal{A}^{\Pi^\varsigma_N})$, \ldots, $\lambda_{k}(\mathcal{A}^{\Pi^\varsigma_N})$, respectively. 
As in the proof of Theorem~\ref{thm-monotonicity}, due to Lemma~\ref{lem-layer-q-form-monotony} the linear operator $\cT_{\theta, \varsigma}\colon \SobNil(\Pi^\varsigma_N)\mapsto\SobNil(\Pi^\theta_N)$, given by formula~\eqref{v_theta},
%$\cT_{\theta, \varsigma} u = \left(\frac{\sin\theta}{\sin\varsigma}\right)^{1/2}u\circ T_{\theta,\varsigma}$, 
is one-to-one and onto and preserves norm.  The functions $v_{j,\theta} = \cT_{\theta,\varsigma}u_j$, $1\leqslant j\leqslant k$, belong to $\SobNil(\Pi_N^\theta)$, remain orthonormal in $L_2(\Pi_N^\theta)$, and hence are linearly independent.  Now, for each $j$ not greater than $k$, we define a subspace $E_j$ as a linear span of the eigenfunctions 
$u_{1}, u_{2}, \ldots, u_{j}$. Since functions $v_{1,\theta}, v_{2,\theta},\ldots, v_{j,\theta}$ belongs to the image $\cT_{\theta,\varsigma}E_j$, the dimension of~$\cT_{\theta,\varsigma}E_j$ equals~$j$.
Moreover, Lemma~\ref{lem-layer-q-form-monotony} results in the following chain of inequalities
\[
\mfa_{\Pi_N^\theta}[\cT_{\theta,\varsigma} u, \cT_{\theta,\varsigma}u] <
\mfa_{\Pi_N^\varsigma}[u, u] \leq \lambda_j(\cA^{\Pi_N^\varsigma})\|u; L_2(\Pi_N^\varsigma)\|^2 = \lambda_j(\cA^{\Pi_N^\varsigma})\|\cT_{\theta,\varsigma} u; L_2(\Pi_N^\theta)\|^2
\]
for all $u\in E_j$.
Applying Lemma~\ref{lem-mutual-ev} to the subspace $\cT_{\theta,\varsigma}E_j$, we obtain the desired inequality for $\lambda_j(\cA^{\Pi_N^\theta})$.
\end{proof}

% Main difference with Theorem~\ref{thm-monotonicity} is that the threshold of $\Pi^\theta_N$ depends on $\theta$. Though we can not guarantee the existence

\begin{remark}
The first eigenvalue $\lambda_1(\mathcal{A}^{\Pi^\theta_N})$ converges to $\pi^2$ as $\theta\to\pi/2$. 
\end{remark}
\begin{proof}
Let $u$ be an eigenfunction normalized in $L_2(\Pi^\theta_N)$ corresponding to $\lambda_1(\mathcal{A}^{\Pi^\theta_N})$. Using the notation from the proof of Lemma~\ref{lem-layer-q-form-monotony}, we have
 \begin{equation*}
    \mfa_{\Pi_N^\theta}[u, u] = \mfa_{\Pi^{\pi/2}_N}[v_{\pi/2}, v_{\pi/2}] + \cos\theta\cdot {\mathfrak r}(v_{\pi/2}).
\end{equation*}
Applying the following arithmetic–geometric mean inequalities
\begin{eqnarray*}
    2|\cos(\varphi-j\psi_N)r\partial_rv\partial_zv| &\leq & r|\partial_r v|^2 + \cos^2(\varphi-j\psi_N)r|\partial_zv|^2, \\
    2|\sin(\varphi-j\psi_N)\partial_\varphi v\partial_zv| &\leq & r^{-1}|\partial_\varphi v|^2 + \sin^2(\varphi-j\psi_N)r|\partial_zv|^2,
\end{eqnarray*}
we obtain an estimate
\[
|{\mathfrak r}(v_{\pi/2})| \leq \mfa_{\Pi^{\pi/2}_N}[v_{\pi/2}, v_{\pi/2}],
\]
which implies
\begin{equation*}
     \mfa_{\Pi^\theta_N}[u,u] \geqslant (1 - \cos\theta)\mfa_{\Pi^{\pi/2}_N}[v_{\pi/2},v_{\pi/2}].
\end{equation*}
Since $\lambda_1(\mathcal{A}^{\Pi^\theta_N}) =  \mfa_{\Pi^\theta_N}[u,u]$, using the max-min principle, we obtain the following chain of inequalities
\begin{equation*}
    \pi^2>\Lambda_\dagger^\theta>\lambda_1(\mathcal{A}^{\Pi^\theta_N})\geqslant (1 - \cos\theta)\mfa_{\Pi^{\pi/2}_N}[v_{\pi/2},v_{\pi/2}]\geqslant(1-\cos\theta)\pi^2.
\end{equation*}
Since $\cos\theta\to0$ as $\theta\to\pi/2$, the squeeze theorem ensures that the first eigenvalue~$\lambda_1(\mathcal{A}^{\Pi^\theta_N})$ converges to $\pi^2$.
\end{proof}

\begin{remark}
\label{rem-5}
    Theorem \ref{thm-layer-regular-mon} remains valid not only for regular layers, but also for all layers that admit an inscribed ball. The proof should be slightly modified in this general case, since the partition $\{\varpi_{N,j}^\theta\}$ then involves unequal parts. To address this, one should replace equal angles $\psi_N$ by unequal angles  $\psi_N^j$ in the definitions of the partition $\{\varpi_{N,j}^\theta\}$ and the  isomorphisms~$T_\theta$ and~$T_\theta^{-1}$.
\end{remark}

% \begin{figure}[ht!]
%     \centering
%     \includegraphics[width=0.4\textwidth]{figure_3.eps}  
%     \caption{The set $\varpi_3$ for the trihedral layer $\Pi$ on  Fig. \ref{fig-01}}
% \end{figure}

\section{Asymmetry influence}
\label{sec-6}

In the previous section, we demonstrated that the unfolding of a regular polyhedral layer, similar to increasing the opening angle of a V-shaped waveguide or a conical layer, leads to analogous behaviour -- the eigenvalues below the threshold of the continuous spectrum depend monotonically on the degree of opening of the layer. It should be pointed out that the significant difference between polyhedral layers and the  aforementioned planar and conical shapes is that not only the spectrum below the threshold but also the threshold value $\Lambda_\dagger$ itself depends on the degree of opening, since $\Lambda_\dagger$ is characterized by the geometry of the layer, in particular, its dihedral angles.
It is the simultaneous changes of the discrete spectrum and the threshold $\Lambda_\dagger$ that is key to the example discussed below.
 
It might appear that symmetry is not essential and serves only to simplify the formalization of the proof (see, in particular, Remark~\ref{rem-5}). However, asymmetry, especially in the rates at which the faces unfold, can lead to unexpected effects: since the threshold $\Lambda_\dagger$ shifts during unfolding of a layer, two spectra of two layers cannot be compared by examining eigenvalue behaviour only.
Below, we present an example of two trihedral layers where one is obtained by unfolding the other, yet the more unfolded layer has eigenvalues below the (shifted) threshold, whereas the less unfolded one does not. We call a trihedral layer more unfolded if all plane and dihedral angles of its outer boundary are not less than the corresponding plane and dihedral angles of the other.

\begin{figure}[ht!]
    \centering
    \includegraphics[width=0.6\textwidth]{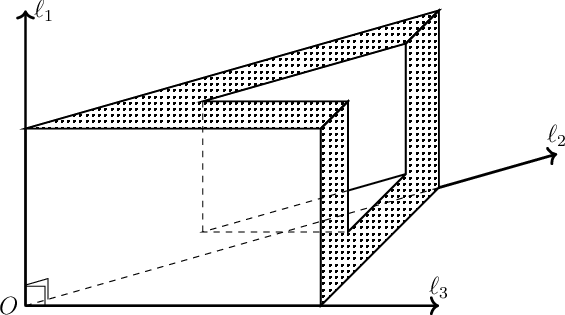}\qquad
    \caption{A layer constructed from a trihedral angle with two right vertex angles and a small third vertex angle}
    \label{fig-03}
\end{figure}

Such an example follows from our work~\cite{BaMa24}, where it was proved that a trihedral layer with two right vertex angles and one sufficiently small vertex angle has no eigenvalues below the threshold. Increase of this small vertex angle
changes the layer geometry, hence, both egenvalues and the threshold value shifts,
and eventually it leads to the emergence of an eigenvalue $\lambda_1<\Lambda_\dagger$. Indeed, in the limmiting case when this angle becomes right, the existence of such an eigenvalue is established in~\cite{BaAI20}.
We conjecture that the observed eigenvalue emmergence/absorption effect is specific for the assymetry case and is not realized for regular layers.

\end{document}

%% file: renames.tex
\renewcommand{\leq}{\leqslant}
\renewcommand{\geq}{\geqslant}

\makeatletter
\newif\ifmeanright
\newcommand{\finite}[1]{\meanrightfalse
\ifx#1H\meanrighttrue\fi
\ifx#1B\meanrighttrue\fi
\ifmeanright{\lefteqn{\overset\circ{\hphantom{#1'}\vphantom{\prime}}}#1}
\else{\lefteqn{\overset\circ{\hphantom{#1}\vphantom{\prime}}}#1}\fi}
\def\Hnil{\@ifnextchar^{\@HnilInd}{\finite H}}
\def\@HnilInd^#1{\finite H\vphantom{H}^{#1}}
\makeatother
\def\SobNil{\finite H\vphantom{H}^1}

\def\EatFirstChar#1{\relax}
\makeatletter
\newcommand{\wt}[1]{\futurelet\@next\@wt#1}
\def\@wt{
\ifx\@next\Fm \widetilde F_-\expandafter\EatFirstChar
\else
  \ifx\@next\Bcondm  \widetilde\Gamma_-
  \else
    \ifx\@next\Fp \widetilde F_+
    \else
      \ifx\@next\Bcondp \widetilde \Gamma_+
      \else
        \ifx\@next\Bcondpd \widetilde \Gamma_+^\perp
        \else\widetilde\@next
        \fi
      \fi
    \fi
  \fi
\expandafter\EatFirstChar\fi
}
\makeatother

\def\bea{\begin{eqnarray}}
\def\eea{\end{eqnarray}}

\def\beq{\begin{equation}}
\def\eeq{\end{equation}}

\def\beal{\begin{align*}}

\def\eeal{ \end{align*} }

\def\bfx{{\bf x}}

\def\cA{{\mathcal A}}

\def\cD{{\mathcal D}}

\def\cT{{\mathcal T}}

\def\mfa{{\mathfrak a}}

\def\SobNil{\finite H\vphantom{H}^1}
\def\SobNilInd^#1{\finite H\vphantom{H}^{#1}}